\documentclass[12pt, reqno]{amsart}
\usepackage{ amsmath, amsthm, amsfonts, amssymb, color}
 \usepackage{mathrsfs}
\usepackage{amsfonts, amsmath}
\usepackage{amsmath}
\usepackage{amsfonts}
\usepackage{amssymb}
\usepackage{times}

\allowdisplaybreaks

\DeclareMathOperator{\sgn}{sgn}

\newtheorem{thm}{Theorem}
\newtheorem{lem}[thm]{Lemma}
\newtheorem{prop}[thm]{Proposition}
\newtheorem{cor}[thm]{Corollary}

\newtheorem{rem}{Remark}

\newcommand{\C}{{\mathbb C}}

\newcommand{\R}{{\mathbb R}}

\newcommand{\cn}{\C^n}

\newcommand\CC{{\mathbb C^n}}

\newcommand\RR{{\mathbb R^n}}
\newcommand\RN{{\mathbb R^n}}

\begin{document}

\title[Integral operators on the fractional Fock-Sobolev spaces]{Boundedness criterion for  
integral operators on the fractional Fock-Sobolev spaces
  }
\author[G. Cao]{Guangfu Cao}
\address{Cao: Department of Mathematics, South China Agricultural University,
Guangzhou, Guangdong 510640, China.}
\email{guangfucao@163.com}
\author[L. He]{Li He*}
\address{He: School of Mathematics and Information Science,
Guangzhou University, Guangzhou 510006, China.}
\email{helichangsha1986@163.com}
\author[J. Li]{Ji Li}
\address{Li:  Department of Mathematics, Macquarie University, NSW, 2109, Australia}
\email{ji.li@mq.edu.au}
\author[M. Shen]{Minxing Shen}
\address{Shen: Department of Mathematics, Sun Yat-sen (Zhongshan)
University, Guangzhou, 510275, P.R. China}
\email{shenmx3@163.com}

\thanks{* Corresponding author, email: helichangsha1986@163.com}

\subjclass[2010]{30H20, 42A38, 44A15}
\keywords{ Fock-Sobolev space,   Hermite-Sobolev space,     integral operator,
Hermite operator, Bargmann transform.}

\begin{abstract}  
We provide a boundedness criterion for the integral operator $S_{\varphi}$ on the fractional Fock-Sobolev space $F^{s,2}(\cn)$, $s\geq 0$, where $S_{\varphi}$ (introduced by  Zhu \cite{Z}) is given by
\begin{eqnarray*} 
S_{\varphi}F(z):= \int_{\mathbb{C}^n} F(w) e^{z \cdot\bar{w}} \varphi(z- \bar{w}) d\lambda(w)
\end{eqnarray*}
with  $\varphi$ in the Fock space  $F^2(\CC)$ and
$d\lambda(w): = \pi^{-n} e^{-|w|^2} dw$
the Gaussian measure on the complex space  $\mathbb{C}^{n}$. This extends the recent result in \cite{CJSWY}.
The main approach is to develop  multipliers on  the fractional Hermite-Sobolev space $W_H^{s,2}(\R^n)$.
\end{abstract}
\maketitle


\section{Introduction}

Let $\mathbb{C}^{n}$ be the complex $n$-dimensional space
with the inner product 
$$
z\cdot \bar{w}=\sum_{j=1}^{n}z_{j}\overline{w_{j}},
\quad z=(z_{1},\cdots ,z_{n}),  
w=(w_{1},\cdots ,w_{n})\in \mathbb{C}^{n}$$
and modulus $|z|=(z\cdot \bar{z})^{\frac{1}{2}}$.
The Fock space ${F}^2(\mathbb{C}^{n})$ is the set of all entire functions $F$ on $\mathbb{C}^{n}$ such that the norm
$$
\|F\|_{{F}^2(\mathbb{C}^{n})}:=\left(\int_{\mathbb{C}^{n}} |F(z)|^2d\lambda(z)\right)^{1\over 2}<\infty,
$$
where
$
d\lambda(z) := \pi^{-n} e^{-|z|^2} dz
$
is the Gaussian measure on $\mathbb{C}^{n}$  (\cite{B1,B2}). 

Let $\mathbb{N}$ be  the set of positive integers  
and $\mathbb{N}_{0}=\mathbb{N}\cup \{0\}.$ For any $\alpha \in \mathbb{N}_{0}^{n}$ we define
$$
e_{\alpha }:=\frac{z^{\alpha }}{\|z^{\alpha }\|_{F^{2}}}=\frac{z^{\alpha }}{\sqrt{\alpha !}}.
$$
Then $\{e_{\alpha }|\alpha \in  \mathbb{N}_{0}^{n}\}$ is an orthonormal basis  for ${F}^2(\mathbb{C}^{n})$.
The Fock space ${F}^2(\mathbb{C}^{n})$ is a Hilbert space with the inner product inherited from $L^2(\mathbb C^n,d\lambda)$.
 The fractional Fock-Sobolev space of order $s\in\R$ is defined by 
$$
F^{s,2}(\CC):=\Big\{f=\sum_{\alpha\in\mathbb{N}_0^n}c_\alpha e_\alpha: \|f\|_{F^{s,2}(\CC)}<\infty\Big\}
$$
with the norm given by 
$$
\|f\|_{F^{s,2}(\CC)}:=\Big[\sum_{\alpha\in\mathbb{N}_{0}^{n}}(2|\alpha|+n)^{s}|c_\alpha|^2
\Big]^{\frac{1}{2}}.
$$
 The Fock-Sobolev  space is a convenient tool for many problems in functional analysis, mathematical physics, and engineering.
We refer to \cite{B1, B2, BC, F, G, Z, Z2} for an introduction. 

In \cite{Z}, Zhu introduced the following integral operator 
\begin{eqnarray}\label{e1.1}
S_{\varphi}F(z):= \int_{\mathbb{C}^n} F(w) e^{z \cdot\bar{w}} \varphi(z- \bar{w}) d\lambda(w),
\end{eqnarray} 
which recovers (or is linked to) many fundamental examples of integral operators in harmonic analysis and complex analysis with different choices of  $\varphi\in   F^2(\CC)$,
including the Riesz transform on $\mathbb R^n$ and the Ahlfors--Beurling operator on $\mathbb C$.
Thus, characterizing the boundedness of $S_\varphi$ is interesting and non-trivial. 
 In 
  \cite{CJSWY},  it was  shown  by Wick, Yan and the first, third and fourth authors
   that the  integral operator $S_\varphi$ in \eqref{e1.1}
  is bounded on $ F^2(\CC)$ if and only if there exists an $m\in L^\infty(\RR)$
such that
\begin{eqnarray}\label{e1.2}
	\varphi(z)=\left({2\over \pi}\right)^{n\over 2}\int_{\RR} m(x) e^{-2\left(x-\frac{i}{2}   z  \right)^2}dx, \ \ \ \  z\in \CC.
\end{eqnarray}
Moreover, we have that
$$ \|S_\varphi\|_{ F^2(\CC)\to F^2(\CC)} = \|m\|_{L^\infty(\mathbb R^n)}. $$

The purpose of the paper is to contine the line in  \cite{CJSWY, Z} to establish 
    a  boundedness criterion for the  integral operator
   $S_{\varphi}$ on the fractional Fock-Sobolev  space $F^{s,2}(\cn)$    by developing the 
    multipliers on  the fractional Hermite-Sobolev space $W_H^{s,2}(\R^n)$
  (see Section 3 below about the  multipliers on  the fractional Hermite-Sobolev space $W_H^{s,2}(\RR)$).
Our main result is the following.

\begin{thm}\label{th1.1}
Let  $s\geq 0$. Then the integral operator $S_{\varphi }$ is bounded on $F^{s,2}(\CC)$ if and only if
there exists  a multiplier $m$ on the space $ W_{H}^{s,2}(\RR) $ such that
$$
\varphi (z)= \left({2\over \pi}\right)^{n\over 2} \int_{\mathbb{R}^{n}}m(x)e^{-2(x-\frac{i}{2}z)^{2}}dx, \qquad z\in \mathbb{C}^{n}.
$$
In particular,  $\varphi \in  F^{s,2}(\mathbb{C}^n)$ if $S_{\varphi }$ is bounded on $F^{s,2}(\mathbb{C}^n)$.
\end{thm}

We would like to mention  that in a recent paper 
\cite{WW},  Wick and Wu   obtained an isometry between the  Fock-Sobolev space and the Gauss-Sobolev space.
As an application, they used multipliers on the  Gauss-Sobolev space to characterize the boundedness of the integral operator  
$S_{\varphi }$  in \eqref{e1.1} on the   Fock-Sobolev spaces $F^{s,2}(\CC)$ when $s$   is a positive integer.
Note that  when  $s$ is a positive integer, our result in Theorem~\ref{th1.1} coincides with their result in \cite[Theorem 4.4]{WW}.

	The layout of the article is as follows.  In Section 2 we prove some properties of 
 the fractional Hermite-Sobolev spaces and  the fractional Fock-Sobolev spaces for proving  our main result. 
 In Section 3 we develop the  multipliers on  the fractional Hermite-Sobolev spaces. The proof of our  Theorem~\ref{th1.1}  will be 
 given in Section 4 by adapting  an argument in \cite{CJSWY} to 
 the fractional Fock-Sobolev space $F^{s,2}(\cn)$  for $s\geq 0$
 by using  multipliers on  the fractional Hermite-Sobolev space $W_H^{s,2}(\R^n)$ in Section 3.

Throughout, the letter ``$c$"  and ``$C$" will denote (possibly
different) constants  that are independent of the essential
variables.

\medskip

\section{Preliminaries}

Let $H$ be the Hermite operator (also called the harmonic oscillator) in the $n$-dimension real space
$\mathbb{R}^{n}$,    which is defined by
  \begin{eqnarray}\label{eq1}
H:=-\Delta + |x|^2 :=-\sum_{i=1}^n {\partial^2\over \partial x_i^2} + |x|^2, \quad x=(x_1, \cdots, x_n).
\end{eqnarray}
 The Hermite operator arises naturally in mathematical physics (see \cite{FH}).  
 For each non-negative integer $k$, 
 the Hermite polynomials $H_{k}$ on  $\mathbb{R}$ are defined by
 $
H_{k}=e^{x^{2}}\frac{d^{k}}{dx^{k}}\big(e^{-x^{2}}\big)   
 $
and by normalization in $L^{2}(\mathbb{R})$,  the Hermite functions
$$
h_{k}(x)=(\sqrt{\pi }2^{k}k!)^{-\frac{1}{2}}e^{-\frac{x^{2}}{2}}(-1)^{k}H_{k}(x),\qquad x\in\R.
$$  
It is not difficult to check that
\begin{equation}\label{eq2}
\left(-\frac{d^{2}}{dx^{2}}+x^{2}\right)h_k(x)=(2k+1)h_k(x).
\end{equation}

In the higher dimensions, for each multi-index
$\alpha =(\alpha_{1},\cdots ,\alpha_{n})\in \mathbb{N}_{0}^{n}$, the Hermite function
$h_{\alpha }$ on $\R^n$ is defined by
$$
h_{\alpha }(x)=\prod_{j=1}^{n}h_{\alpha_{j}}(x_{j}), \hskip 5mm x=(x_{1},\cdots , x_{n})\in \mathbb{R}^{n}.
$$
By (\ref{eq2}), we see that 
\begin{equation}\label{eq5}
Hh_{\alpha}=(2|\alpha |+n)h_{\alpha }.
\end{equation}
That is,  these $\{h_{\alpha }\}$ are eigenfunctions of the Hermite operator $H$. 
Moreover, $\{h_{\alpha }\}_{\alpha \in \mathbb{N}_{0}^{n}}$ is an orthonormal basis of
$L^{2}(\mathbb{R}^{n}).$
Note there is a constant $C>0$ such that $\|h_{\alpha }\|_{L^{\infty }(\mathbb{R}^{n})}\leq C$ for all
$\alpha \in \mathbb{N}_{0}^{n}$, and for each $m\in \mathbb{N}$, we have
$$
|\langle f, h_{\alpha }\rangle_{L^{2}(\mathbb{R}^{n})}|\leq \|H^{m}f\|_{L^{2}(\mathbb{R}^{n})}
(2|\alpha |+n)^{-m}.
$$
Hence, if $f$ is a rapidly decreasing function, then the Hermite series expansion
$$
f=\sum_{\alpha \in \mathbb{N}_{0}^{n}}\langle f, h_{\alpha }\rangle_{L^{2}(\mathbb{R}^{n})}h_{\alpha }
$$
converges to $f$ uniformly in $\mathbb{R}^{n}$, and certainly, also in $L^{2}(\mathbb{R}^{n})$.

Let $s\in \R$ and $f\in \mathscr S(\R^n)$. One defines the fractional Hermite operator $H^s$ by
$$
	H^{s}f:=\sum_{\alpha \in \mathbb{N}_{0}^{n}}(2|\alpha |+n)^{s}\langle f, h_{\alpha }
\rangle_{L^{2}(\mathbb{R}^{n})}h_{\alpha }.
$$
The fractional Hermite-Sobolev space of order $s\in\R$ is defined by
$$
	W_{H}^{s,2}(\RR):= \Big\{ f\in L^2(\RN): H^{\frac{s}{2}}f\in L^{2}(\mathbb{R}^{n}) \Big\}
$$
with the norm given by $\|f\|_{W_{H}^{s,2}(\R^n)}:= \|H^{\frac{s}{2}}f\|_{L^{2}(\mathbb{R}^{n})}$ 
(see \cite{BT, CCL}).

For $1\leq j\leq n$,  let
$$
H_{j}:=\frac{\partial }{\partial x_{j}}+x_{j}\hskip 5mm\mbox{and}\hskip 5mm H_{-j}:=H_{j}^{*}
=-\frac{\partial }{\partial x_{j}}+x_{j}.
$$
Then it is easy to check that
$$
H=\frac{1}{2}\sum_{j=1}^{n}[H_{j}H_{-j}+H_{-j}H_{j}].
$$
For any positive integer $k$ we define $\widetilde{W}_H^{k,2}(\RR)$ as the space of functions $f\in L^{2}(\mathbb{R}^{n})$
such that for every $1\leq |j_{1}|,\cdots ,|j_{m}|\leq n $ and $ 1\leq m\leq k,$
$$
H_{j_{1}}\cdots H_{j_{m}}f\in L^{2}(\mathbb{R}^{n}).
$$
The norm on $\widetilde{W}_H^{k,2}(\RR)$ is given by
$$
\|f\|_{\widetilde{W}_H^{k,2}(\RR)}:=\sum_{1\leq |j_{1}|,\cdots ,|j_{m}|\leq n, 1\leq m\leq k}\|H_{j_{1}}\cdots
H_{j_{m}}f\|_{L^{2}(\RR)}+\|f\|_{L^{2}(\RR)}.
$$

\begin{lem}[\cite{BT} Theorem 4]
For $k\in \mathbb{N}$ we have that  $\widetilde{W}_H^{k,2}(\RR)=W_{H}^{k,2}(\RR)$. Moreover,
the norms $\|\cdot \|_{\widetilde{W}_H^{k,2}(\RR)}$ and $\|\cdot \|_{W_{H}^{k,2}(\RR)}$ are equivalent.
\label{lem1}
\end{lem}

For general integer $s\geq 1$, we can also characterize $W_H^{s,2}(\RR)$ with the help of operators $H_j$.
\begin{lem}\label{lem qq}
	For $s\geq 1$, there holds
\begin{equation}\label{c3}
	\|f\|_{W_H^{s,2}(\RR)}\approx \sum_{1\leq |j|\leq n} \|H_jf\|_{W_H^{s-1,2}(\RR)} + \|f\|_{L^2(\RR)}.
\end{equation}
\end{lem}

\begin{proof}
	It is obvious that (\ref{c3}) holds for $s\in \mathbb{N}$ by Lemma~{\ref{lem1}}. Assume $k<s<k+1$, where $k\geq 1$ is an integer. 
	Applying (\ref{c3}) to the norm $\|\cdot\|_{W_H^{k,2}(\RR)}$ indicates
$$
	\|f\|_{W_H^{s,2}(\RR)}= \|H^{s-k\over 2}f \|_{W_H^{k,2}(\RR)}\approx \sum_j \|H_jH^{s-k\over 2}f\|_{W_H^{k-1,2}(\RR)} + \|f\|_{L^2(\RR)}.
$$ 

It follows from Lemma 4 in \cite{BT} that
\begin{equation}\label{c1}
	H_jH^{s-k\over 2}f= (H+2)^{s-k\over 2}H_jf,\quad 1\leq j\leq n,
\end{equation}
and 
\begin{equation}\label{c2}
	H_jH^{s-k\over 2}f= (H-2)^{s-k\over 2}H_jf,\quad -n\leq j\leq -1.
\end{equation}
By (\ref{c1}), (\ref{c2}) and the definition of the ${W_H^{k,2}(\RR)}$-norm, we have
\begin{align*}
	\|H_jH^{s-k\over 2}f\|_{W_H^{k-1,2}(\RR)}^2
&= \|(H\pm2)^{s-k\over 2}H_jf\|_{W_H^{k-1,2}(\RR)}^2  \\
&=  \|H^{k-1\over 2}(H\pm2)^{s-k\over 2}H_jf\|_{L^2(\RR)}^2  \\
&= \sum_\alpha (2|\alpha|+n)^{k-1}(2|\alpha|+n\pm 2)^{s-k} |\langle H_jf,h_\alpha \rangle|^2  \\
&\approx \sum_\alpha (2|\alpha|+n)^{s-1}|\langle H_jf,h_\alpha \rangle|^2  \\
&= \|H_jf\|_{W_H^{s-1,2}}^2(\RR).
\end{align*}
Thus, $$\|f\|_{W_H^{s,2}(\RR)}\approx \sum_j \|H_jf\|_{W_H^{s-1,2}(\RR)}+ \|f\|_{L^2(\RR)}.$$
The proof of Lemma \ref{lem qq} is complete.
\end{proof}

For any $a=(a_1, a_2, \cdots, a_n)\in\R^n$ we use $\tau_a$ to denote the operator of translation by $a$, namely,
$\tau_af(x)=f(x-a)$.

\begin{lem}
Let $m\in \mathbb{N}$, $a\in \mathbb{R}$, and $f\in L^{2}(\RR)$. Then
$$
H_{j_{1}}\cdots H_{j_{m}}\tau_{a}f(x)=\sum_{l\leq m}p_{l}(a)\tau_{a}H_{j'_{1}}\cdots H_{j'_{l}}f(x),
$$
where
$$
1\leq |j_{1}|,\cdots ,|j_{m}|\leq n,1\leq |j'_{1}|,\cdots ,|j'_{l}|\leq n,
$$
and $p_{l}(\cdot)$ is a polynomial of order $l, \ 1\leq l\leq m.$
\label{lem2}
\end{lem}

\begin{proof}
For $1\leq |j|\leq n$,
\begin{eqnarray*}
H_{j}\tau_{a}f(x)&=&\pm \frac{\partial}{\partial x_{|j|}}f(x-a)+x_{|j|}f(x-a)\\
&=&\pm\frac{\partial}{\partial (x-a)_{|j|}}f(x-a)\\
&&\qquad+(x_{|j|}-a_{|j|})f(x-a)+a_{|j|}f(x-a)\\
&=&H_{j}f(x-a)+a_{|j|}f(x-a)\\
&=&\tau_{a}[H_{j}+a_{|j|}]f(x).
\end{eqnarray*}
By mathematical induction, we obtain the desired result.
\end{proof}

The  Bargmann transform is the classical tool in mathematics analysis and mathematical physics
(see \cite{B1, B2, F, G, Z2}  and references therein).
   Consider $f\in L^2(\RR)$, and define
\begin{align}\label{e2.1}
 \mathcal{B} f(z)&:= \left({2\over \pi}\right)^{n\over 4}\int_{\RR} f(x) e^{2x\cdot z-x^2- {z^2\over 2}} dx\nonumber \\
 &= \left({2\over \pi}\right)^{n\over 4} e^{ {z^2\over 2}}  \int_{\RR} f(x) e^{-(x-z)^2}dx, \ \ \ \ z\in \CC.
\end{align}
   For $s\geq 0$
it is clear that $W_{H}^{s,2}(\RR)\subset L^{2}(\RR)$, so $\mathcal{B}f$ is well-defined on $W_{H}^{s,2}(\RR)$. Also, it
is well known that, for
$$
f=\sum_{\alpha  \in \mathbb{N}_{0}^{n}}c_\alpha h_{\alpha },
$$
we have
$$
\mathcal{B}f=\sum_{\alpha\in\mathbb{N}_{0}^{n}}c_\alpha
\mathcal{B}h_{\alpha }=\sum_{\alpha\in\mathbb{N}_{0}^{n}}c_\alpha e_{\alpha }.
$$
Consequently, the Bargmann transform is a unitary operator from $L^2(\R^n)$ to $F^2(\CC)$, and it is also
a unitary operator from the fractional Hermite-Sobolev space $W_H^{s,2}(\RR)$ to the fractional Fock-Sobolev
space $F^{s,2}(\CC)$.

There is an equivalent definition for the Fock-Sobolev spaces, that is, so called  the weighted Fock spaces. Given a real
number $s$ we define $F_{s}^{2}(\CC)$ as the space of entire functions $f$ on $\cn$ with
$$
\|f\|_{F_{s}^{2}(\CC)}^{2}=\omega_{n,s}\int_{\mathbb{C}^{n}}|(1+|z|)^{2s}|f(z)|^{2}e^{-|z|^{2}}\,dz
<\infty,
$$
where $\omega_{n,s}$ is a normalizing constant such that the constant function 1 has norm 1.
It follows from Lemma~\ref{lem3} below that the fractional Fock-Sobolev spaces are the same as these
weighted Fock spaces whose definition does not involve derivatives. Sometimes, it is more convenient to
study function theoretic and operator theoretic properties on the weighted Fock spaces instead of the
Fock-Sobolev spaces.

\begin{lem}[\cite{CP} Theorem1.2]
For $s\in \mathbb{R}$ we have $F^{s,2}(\CC)=F_{s}^{2}(\CC)$ with equivalent norms.
\label{lem3}
\end{lem}

Recall that the Weyl operators $W_a$, $a\in\cn$, are defined by
$$W_af(z):=f(z-a)e^{-\frac{|a|^2}2+z\cdot\bar a}.$$

\begin{lem}\label{lem4}

Suppose $s\in\R$ and $a\in\cn$. Then $W_{a}$ is bounded on $F^{s,2}(\CC)$. Moreover, for all $f\in F^{s,2}(\CC)$
$$
	\|W_a f\|_{F^{s,2}(\CC)} \leq  C (1+|a|^{|s|})\,\|f\|_{F^{s,2}(\CC)}.
$$
\end{lem}
\begin{proof}
By Lemma \ref{lem3}, we have that for any $f\in F^{s,2}(\CC)$, 
\begin{align*}
\|W_{a}f\|_{F^{s,2}(\CC)}^{2}&\leq C\|W_{a}f\|_{F^2_s(\CC)}^{2}\\
&=C\int_{\mathbb{C}^{n}}(1+|z|)^{2s}|f(z-a)|^{2}e^{-|z-a|^{2}}\,dz\\
&=C\int_{\mathbb{C}^{n}}(1+|z+a|)^{2s}|f(z)|^{2}e^{-|z|^{2}}\,dz.
\end{align*}
Note that
 $
	1+|z+a|\leq 1+|z|+|a|\leq (1+|z|)(1+|a|). 
 $
Also 
 $
	1+|z|= 1+|(z+a)-a| \leq (1+|z+a|)(1+|a|) 
 $
and so 
 $
	(1+|z+a|)^{-1}\leq (1+a)(1+|z|)^{-1}.
 $
Thus for $s\in \R$, 
\begin{align*}
\|W_{a}f\|_{F^{s,2}(\CC)}^{2}
&\leq C(1+|a|)^{2|s|} \int_{\cn}(1+|z|)^{2s}|f(z)|^2e^{-|z|^2}\,dz\\
&\leq C(1+|a|)^{2|s|} \|f\|^2_{F^{s,2}(\CC)}.
\end{align*}
The proof of Lemma \ref{lem4} is complete.
\end{proof}

Since $\tau_{a}=\mathcal{B}^{-1}W_{a}\mathcal{B}$ and the Bargmann transform is a unitary
operator from $W_H^{s,2}(\RR)$ to $F^{s,2}(\CC)$, we see that $\tau_{a}$ is bounded on $W_{H}^{s,2}(\RR)$ 
for each $a\in \mathbb{R}^{n}$ and 
\begin{eqnarray}\label{ppp}
\|\tau_{a}f\|_{W_{H}^{s,2}(\RR)}\leq C (1+|a|^{|s|})\,\|f\|_{W_{H}^{s,2}(\RR)}
\end{eqnarray}
for all $f\in W_{H}^{s,2}(\RR)$.
A direct computation shows that $S_{\varphi }$ commutes with $W_{a}$ on $F^{s,2}(\CC)$, that is,
$S_{\varphi }W_{a}=W_{a}S_{\varphi }$; see \cite{CJSWY}. Since $\mathcal{B}W_{H}^{s,2}(\RR)=F^{s,2}(\CC)$
and $\mathcal{B}h_{\alpha }=e_{\alpha }$, we see that $T=\mathcal{B}^{-1}S_{\varphi }\mathcal{B}$ 
commutes with $\tau_{a}$ for $a\in \mathbb{R}^{n}$.  In fact, for any $f\in W_{H}^{s,2}(\RR)$,
\begin{eqnarray*}
\tau_{a}Tf(z)&=&\tau_{a}\mathcal{B}^{-1}S_{\varphi }\mathcal{B}f\\
&=&(\mathcal{B}^{-1}W_{a}\mathcal{B})(\mathcal{B}^{-1}S_{\varphi }\mathcal{B})f\\
&=&(\mathcal{B}^{-1}W_{a}S_{\varphi }\mathcal{B})f\\
&=&(\mathcal{B}^{-1}S_{\varphi }W_{a}\mathcal{B})f\\
&=&(\mathcal{B}^{-1}S_{\varphi }\mathcal{B})(\mathcal{B}^{-1}W_{a}\mathcal{B})f\\
&=&T\tau_{a}f(z).
\end{eqnarray*}

In the following, the Fourier transform of a function $f$ is given by 
$$
{\mathcal F} f(x) :=\pi^{-{n\over2}} \int_{\RR} e^{-2ix\cdot y}f(y) dy, \ \ \ \ x\in\RR,
$$
The inverse of the Fourier  transform ${\mathcal F}$ will be denoted by ${\mathcal F}^{-1}$, i.e,
 ${\mathcal F}{\mathcal F}^{-1}={\mathcal F}^{-1}{\mathcal F}=Id$, the identity operator on $L^2(\RR)$.
 Then we have

\begin{lem}
For any $s\in \mathbb{R}$, the Fourier transformation $\mathcal{F}$ is a unitary operator on $W_{H}^{s,2}(\RR)$.
\label{lem5}
\end{lem}

\begin{proof}
Following the proof of Lemma 2.3 in \cite{CJSWY}, we get
$$
\mathcal{B}\mathcal{F}\mathcal{B}^{-1}f(z)=f(-iz), \hskip 5mm\mbox{and}\hskip 5mm  \mathcal{B}\mathcal{F}^{-1}
\mathcal{B}^{-1}f(z)=f(iz).
$$
It is obvious that the operators $f(z)\mapsto f(iz)$ and $f(z)\mapsto f(-iz)$ are unitary on $F^{s,2}(\CC)$.
Since the Bargmann transform is a unitary from $W_H^{s,2}(\RR)$ to $F^{s,2}(\CC)$, we conclude that the
Fourier transform $\mathcal{F}$ is also unitary on $W_H^{s,2}(\RR)$.
\end{proof}

\begin{lem}[\cite{BT} Lemma 3]
Let $p\in [1, \infty )$ and $s>0$. Then the operator
$|x|^{2s}H^{-s}$ is bounded on $L^{p}(\mathbb{R}^{n})$, that is, there is a positive constant $M$ such that
$$
\int_{\mathbb{R}^{n}}||x|^{2s}H^{-s}f(x)|^{p}dx\leq M\int_{\mathbb{R}^{n}}|f(x)|^{p}dx
$$
for all $f\in L^{p}(\mathbb{R}^{n})$.
\label{lem6}
\end{lem}

\medskip

\section{Multipliers on the fractional Hermite-Sobolev  spaces}

In the following we denote $\mathcal{M}(W_{H}^{s,2}(\RR))$
the space of multipliers on $ W_{H}^{s,2}(\RR)$; {\it i.e.} , the set of all functions $g$ such that 
$$
\|g f\|_{ W_{H}^{s,2}(\RR)}\leq M \|f \|_{ W_{H}^{s,2}(\RR)} 
$$
for all $f\in  W_{H}^{s,2}(\RR)$, equipped with the norm defined as the infinium
 of all such $M$ that the above inequality holds. 

Before we discuss the multipliers, we need some characterizations for Hermite-Sobolev functions at first.

 Fix a $K\in{\mathbb N}$, we define
\begin{eqnarray}\label{e1.2}
G_{s,K,H}(f)(x)=\left( \int_0^{\infty} |\big(I -e^{-t^2H}\big)^K f(x)|^2 {dt\over t^{1+2s}} \right)^{1/2}.
\end{eqnarray}
For simplicity, we will often write $G_{s,H}$ in place of $G_{s,1,H}.$
Then we have 
  
\begin{lem}\label{le1.1} If $0< s<2K$, then for  $f\in W^{s, 2}_H(\RN)$
\begin{eqnarray*} 
\|f\|_{W^{s, 2}_H(\RN)} \sim 
\|G_{s,K, H}(f)\|_{L^{  2}(\RN)} 
\end{eqnarray*}
with the implicit equivalent positive constants independent of $f.$
\end{lem}
 
 \begin{proof} It suffices to show that for $0< s<2K$,
 $$
 C^{-1}\|H^{s/2}f\|_{L^{  2}(\RN)} \leq \|G_{s,K, H}(f)\|_{L^{  2}(\RN)} \leq C\|H^{s/2}f\|_{L^{  2}(\RN)}.
  $$
  Denote by $\psi(z)=z^{-s}(1-e^{-z^2})^K$.
It follows from  the spectral theory (\cite{Yo}) that for any $f\in L^2(\RN),$
\begin{align}
\|G_{s,K, H}(H^{-s/2}f)\|_{L^{  2}(\RN)} &=
\Big\{\int_0^{\infty}\|\psi(t\sqrt{H})f\|_{L^2(X)}^2{dt\over t}\Big\}^{1/2}\nonumber\\
&=\Big\{\int_0^{\infty}\big\langle\,\overline{ \psi}(t\sqrt{H})\,
\psi(t\sqrt{H})f, f\big\rangle {dt\over t}\Big\}^{1/2}\nonumber\\
&=\Big\{\big\langle \int_0^{\infty}|\psi|^2(t\sqrt{H}) {dt\over t}f,
f\big\rangle\Big\}^{1/2}\nonumber\\
&\leq \kappa \|f\|_{L^2(\R^n)},
\label{e3.133}
\end{align} 
 where
$\kappa:=\big\{\int_0^{\infty}|{\psi}(t)|^2 {dt/t}\big\}^{1/2}$.
This shows that $\|G_{s,K, H}(f)\|_{L^{  2}(\RN)} \leq C\|H^{s/2}f\|_{L^{  2}(\RN)}.$

On the other hand, from  the spectral theory (\cite{Yo}) we see that for any $f\in L^2(\RN),$
$$
f=c\int_{-\infty}^{\infty} (t^2H)^{-s}  (1-e^{-t^2H})^{2K} (f) {dt\over t}
$$
for some constant $c>0$, where the integral converges in $L^2(\RN)$.
 Hence for every $g\in L^2(\RN)$ with $\|g\|_{L^2(\RR)}\leq 1$, we apply \eqref{e3.133} to get 
\begin{align*}
\left|\langle f, g\rangle \right| &= c\left| \int_{-\infty}^{\infty} \langle (t^2H)^{-s}  (1-e^{-t^2H})^{2K} f,  g\rangle    {dt\over t}\right|
 \nonumber\\
&= c\left| \int_{-\infty}^{\infty} \langle (t^2H)^{-s/2}  (1-e^{-t^2H})^{ K} f, 
(t^2H)^{-s/2}  (1-e^{-t^2H})^{ K}g\rangle    {dt\over t}\right|\nonumber\\
&\leq c \|G_{s,K, H}(H^{-s/2}f)\|_{L^{  2}(\RN)} \|G_{s,K, H}(H^{-s/2}g)\|_{L^{  2}(\RN)} 
\nonumber\\
&\leq C \|G_{s,K, H}(H^{-s/2}f)\|_{L^{  2}(\RN)},  
\end{align*}
which shows that $\|H^{s/2}f\|_{L^{  2}(\RN)} \leq C\|G_{s,K, H}(f)\|_{L^{  2}(\RN)}.$
This proves the lemma.
 \end{proof}
 
 In the following, we let $\eta(x)\in C_0^{\infty}(\RN)$ satisfy
 
 \smallskip
 
i) $0\leq \eta(x)\leq 1,$
 
  \smallskip
 
ii) $\eta(x)=1$ on the cube $\{ |x|\leq 1\}$,

 \smallskip
 
iii) $\eta(x)=0$ outside the cube  $\{ |x|\leq 2\}$.

  \smallskip
 
 iv) $\sum\limits_{m\in {\mathbb Z}_n} \eta_m(x)=c_0$ for some constant $c_0$ and all   $x\in \RN$, where  $\eta_m(x)=\eta(x+m)$  and
 ${\mathbb Z}_n$ denotes the lattice points in $\RN$.

\begin{prop}\label{prop2.9}
Let $s\geq 0$.
The function $f$ belongs to $  W^{s, 2}_H(\RN)$
if and only if $f \eta_m\in  W^{s, 2}_H(\RN) $ for every $m\in {\mathbb Z}_n$
and
\begin{eqnarray}\label{e 2.9}
\bigg( \sum_{m\in {\mathbb Z}_n} \| f \eta_m\|_{W^{s, 2}_H(\RN)}^2 \bigg)^{1/2}<\infty,
\end{eqnarray}
in which case this expression is equivalent to $\|f\|_{W^{s, 2}_H(\RN)}.$

Moreover, the necessarity holds for any  function $\eta$ in $C_0^{\infty}(\RN)$.
\end{prop}

\begin{proof} 
The case $s=0$ is trivial. Let us consider the case $0<s<1$. For $f\in L^2(\RN)$, define
\begin{align*} 
G^{(1)}_{s,  H} f(x)&:=\left( \int_0^{1} |(I-e^{-t^2H}) f(x)|^2 {dt\over t^{1+2s}} \right)^{1/2}\\ 
G^{(2)}_{s,  H} f(x)&:=\left( \int_1^{\infty} |(I-e^{-t^2H}) f(x)|^2 {dt\over t^{1+2s}} \right)^{1/2}.
\end{align*}
and so $G_{s,  H} f(x)\leq G^{(1)}_{s,  H} f(x) + G^{(2)}_{s,  H} f(x).$
Since the kernel $ K_{e^{-t^2H}}(x, y)$ of $e^{-t^2H}$  satisfies 
\begin{eqnarray}\label{kkk}
\Big|K_{e^{-t^2H}}(x, y)\Big| \leq C t^{-n} e^{-{|x-y|^2\over t^2}},
\end{eqnarray}
we see that  
\begin{eqnarray}\label{ttt}
\|G^{(2)}_{s,  H} f\|_{L^2(\RN)}\leq C\|f\|_{L^2(\RN)}.
\end{eqnarray}

 Assume that $f  \eta_m\in  W^{s, 2}_H(\RN) $ for every $m\in {\mathbb Z}_n$
and \eqref{e 2.9} holds.  Let us prove that $f  \in  W^{s, 2}_H(\RN)$  and 
\begin{eqnarray}\label{nn}
\|f\|_{W^{s, 2}_H(\RN)}\leq  C\left( \sum_{m\in {\mathbb Z}_n} \| f \eta_m \|_{W^{s, 2}_H(\RN)}^2 \right)^{1/2}.
\end{eqnarray}

We now prove \eqref{nn}.
From \eqref{ttt}, we use  the properties i), ii) and iii)  of $\eta$ to obtain
$$ 
 \|G^{(2)}_{s,  H} f\|^2_{L^2(\RN)}\leq C\|f \|^2_{L^2(\RR)} \leq C\sum\limits_{m\in {\mathbb Z}_n} \|f\eta_m\|^2_{L^2(\RR)}.
$$ 
Now, let $B_m:=\{x: \ |x-m|\leq 4\}$.   Since $\sum\limits_{m\in {\mathbb Z}_n} \eta_m(x)=c_0$ for all   $x\in \RN$, we see that 
$
  \|G^{(1)}_{s,  H} f\|^2_{L^2(\RN)}\leq (E+F)/c_0^2,
$
where 
$$
E:=  \int_0^{1}\int_{\RN} |\sum_{m\in{\mathbb Z}_n }(I-e^{-t^2H}) (f\eta_m)(x) \chi_{B_m}(x) |^2 {dxdt\over t^{1+2s}} 
 $$
and 
$$
F:= 
\int_0^{1}\int_{\RN} |\sum_{m\in{\mathbb Z}_n } e^{-t^2H}  (f\eta_m)(x) \chi_{B^c_m}(x) |^2 {dxdt\over t^{1+2s}} $$
For the term $E$, we have 
\begin{eqnarray*} 
E&\leq&  \sum_{m\in{\mathbb Z}_n } \int_0^{1}\int_{\RN} |(I-e^{-t^2H}) (f\eta_m)(x)   |^2 {dxdt\over t^{1+2s}} \\ 
 &\leq& C \sum_{m\in {\mathbb Z}_n} \| f \eta_m\|_{W^{s, 2}_H(\RN)}^2.  
\end{eqnarray*}
To estimate the term $F$, we note that 
\begin{eqnarray*}  
F
 &\leq&     \int_0^{1} \int_{\RN} \Big|\int_{\RN} |K_t(x,y) | \Big|\sum_{m}  f\eta_m (y)\Big| dy\Big|^2 dx  {dt\over t^{1+2s}}, \\
\end{eqnarray*}
where
$$
|K_t(x,y)|\leq C t^{-n} e^{-c{|x-y|^2\over t^2}} \chi_{\{|x-y|\geq 2\}}.
$$
 It follows that $F\leq C\sum\limits_{m\in {\mathbb Z}_n} \| f\eta_m\|_{L^2(\RR)}^2  $. Hence, estimates $E$ and $F$ yield that 
 $ \|G^{(1)}_{s,  H} f\|^2_{L^2(\RN)} \leq C\sum\limits_{m\in {\mathbb Z}_n} \| f \eta_m \|_{W^{s, 2}_H(\RN)}^2$.
This completes  the proof of the sufficiency part in the case $0<s<1$.
 
 Now we  prove the necessarity part of the proposition under a weaker condition on $\eta\in C_0^{\infty}(\RN)$. 
 Indeed, for every $f\in W^{s, 2}_H(\RN)$  we will show that if 
$$
  \eta \in C_0^{\infty}(\RN)  \ \ \ {\rm and}\ \ \ \sum\limits_{m\in {\mathbb Z}_n} |\eta_m(x)| \leq C, \ \ \ {\rm for\ all\ } x\in \RN,  
  $$
 then   $f\eta_m\in  W^{s, 2}_H(\RN) $ for every $m\in {\mathbb Z}_n$
 and 
 \begin{eqnarray}\label{eta}
\bigg( \sum_{m\in {\mathbb Z}_n} \| f \eta_m\|_{W^{s, 2}_H(\RN)}^2 \bigg)^{1/2}\leq C \| f  \|_{W^{s, 2}_H(\RN)}.
\end{eqnarray}  
 
 To prove \eqref{eta}, we  write 
\begin{eqnarray*} 
 \left( \int_0^{\infty} |\big(I -e^{-t^2H}  \big) (f\eta_m)(x)|^2 {dt\over t^{1+2s}} \right)^{1/2}  
 \leq  I_m(x) +J^{(1)}_m(x) + J^{(2)}_m(x),
\end{eqnarray*}
where
\begin{eqnarray*} 
I_m(x)&:=&\left( \int_0^{\infty} \big|\big(I -e^{-t^2H}  \big)f(x) \eta_m(x) \big|^2 {dt\over t^{1+2s}} \right)^{1/2},\\
J^{(1)}_m(x)&:=& \left( \int_0^{1} | e^{-t^2H} (f\eta_m)(x)-e^{-t^2H} (f)(x)\eta_m(x) |^2 {dt\over t^{1+2s}} \right)^{1/2},\\ 
J^{(2)}_m(x)&:=& \left( \int_1^{\infty} | e^{-t^2H} (f\eta_m)(x)-e^{-t^2H} (f)(x)\eta_m(x) |^2 {dt\over t^{1+2s}} \right)^{1/2}. 
\end{eqnarray*}
Now 
\begin{eqnarray*} 
 &&\hspace{-1cm}\sum_{m\in {\mathbb Z}_n} \int_{\RN} |I_m(x)|^2 dx\\
 &\leq& 
 \int_{\RN}\int_0^{\infty} \big|\big(I -e^{-t^2H}  \big)f(x) \big|^2  
 \left(\sum_{m\in {\mathbb Z}_n} |\eta_m(x) |^2\right) {dt\over t^{1+2s}}  dx \\
 &\leq& 
  C\|f\|_{W^{s, 2}_H(\RN)}.
\end{eqnarray*}
For the term $J^{(1)}_m$,  we apply \eqref{kkk} to get 
 \begin{align*}  
  &\sum_{m\in {\mathbb Z}_n} \int_{\RN} |J^{(1)}_m(x)|^2 dx  \\
 &\leq \sum_{m\in {\mathbb Z}_n} \int_{\RN}  \int_0^{1}  \bigg( \int_{\RN}  t^{-n} e^{-{|y|^2\over t^2}}  
 \left|f(x+y)(\eta_m(x+y)-\eta_m(x))\right| dy\bigg)^2 {dt\over t^{1+2s}} dx.
\end{align*}
For each $x$ the inner is non-zero for at most $7^n$ distanct $m$'s; namely when $|x_i +m_i|\leq 3$. Fix one such $m_i$.
By the mean value theorem $|\eta_m(x+y)-\eta_m(x)|=|y| |\nabla\eta(x_0)|\leq M |y|$ where $M=\|\nabla\eta\|_{L^\infty(\RR)}$. Hence,
 \begin{eqnarray*}  
  &&\hspace{-1cm}\sum_{m\in {\mathbb Z}_n} \int_{\RN}  |J^{(1)}_m(x)|^2 dx  \\
 &\leq& C   \int_{\RN} 
 \int_0^{1}  \left( \int_{\RN}  t^{-n} e^{-{|y|^2\over t^2}}  |f(x+y) | \left({|y|\over t}\right)  dy\right)^2 {dt\over t^{2s-1}} dx\\ 
  &\leq& C \|f\|_{L^2(\RR)}^2 
\end{eqnarray*} 
when $0\leq s<1$. 
Further, we use the property \eqref{kkk} of  the kernel $ K_{e^{-t^2H}}(x, y)$  to see that 
$ 
  \sum_{m\in {\mathbb Z}_n} \int_{\RN}  |J^{(2)}_m(x)|^2 dx  
 \leq  \|f\|_{L^2(\RR)}^2. 
$
Estimates of $I_m$, $J^{(1)}_m$ and $ J^{(2)}_m$ together, give the proof of the necessarity part in the case $0<s<1$.

Thus the proposition is proved for $0\leq s<1.$
For $s\geq 1$, we will prove it by induction.  Suppose the proposition is true for $k\leq s< k+1,$ where $k$ is an integer.  Let $k+1\leq s< k+2.$ 
By Lemma~\ref{lem qq} and the assumption
\begin{align}\label{c4}
	\|f\|_{W_H^{s,2}(\RR)}^2
&\approx \sum_j \|H_jf\|_{W_H^{s-1,2}(\RR)}^2+\|f\|_{L^2(\RR)} \notag \\
&\approx \sum_j\sum_{m\in\mathbb{Z}^n}\|(H_jf)\eta_m\|^2_{W_H^{s-1,2}(\RR)}+\|f\|_{L^2(\RR)}.
\end{align}
A direct calculation shows
\begin{equation}
	H_j(f\eta_m)= (H_jf)\eta_m\pm f\partial_j\eta_m.
\end{equation}
Hence, (\ref{c4}) is controlled by
\begin{align*}
& C\bigg( \sum_j\sum_{m\in\mathbb{Z}^n} \|H_j(f\eta_m)\|^2_{W_H^{s-1,2}(\RR)}+ \sum_j\sum_{m\in\mathbb{Z}^n} \|f\partial_j\eta_m\|_{W_H^{s-1,2}(\RR)}^2\\
&\qquad+\|f\|_{L^2(\RR)} \bigg)\\
&\leq C\bigg( \sum_{m\in\mathbb{Z}^n} \|f\eta_m\|^2_{W_H^{s,2}(\RR)}+  \|f\|_{W_H^{s-1,2}(\RR)}^2 \bigg)\\
&\leq C\bigg( \sum_{m\in\mathbb{Z}^n} \|f\eta_m\|^2_{W_H^{s,2}(\RR)}+  \sum_{m\in\mathbb{Z}^n} \|f\eta_m\|^2_{W_H^{s-1,2}(\RR)}  \bigg) \\
&\leq C \sum_{m\in\mathbb{Z}^n} \|f\eta_m\|^2_{W_H^{s,2}(\RR)},
\end{align*}
where we use Lemma~\ref{lem qq} and the assumption again. Thus the sufficiency is proved for $k+1\leq s<k+2$.

On the other hand, for every $\eta$  in  $C_0^\infty(\mathbb R^n)$, we obtain that
\begin{align*}
&	\sum_{m\in\mathbb{Z}^n} \|f\eta_m\|^2_{W_H^{s,2}(\RR)} \\
&\approx \sum_j \sum_{m\in\mathbb{Z}^n} \|H_j(f\eta_m)\|^2_{W_H^{s-1,2}(\RR)}   \\
&\leq C\bigg(  \sum_j \sum_{m\in\mathbb{Z}^n} \|(H_jf)\eta_m\|^2_{W_H^{s-1,2}(\RR)}+ \sum_j \sum_{m\in\mathbb{Z}^n} \|f\partial_j\eta_m\|^2_{W_H^{s-1,2}(\RR)}  \bigg)  \\
&\leq C\bigg( \sum_j \|H_jf\|_{W_H^{s-1,2}(\RR)}^2+  \|f\|_{W_H^{s-1,2}(\RR)}^2  \bigg)  \\
&\leq C\|f\|_{W_H^{s,2}(\RR)}^2.
\end{align*}
This shows Proposition~~\ref{prop2.9} also holds for $k+1\leq s< k+2$. 
\end{proof}

\begin{lem}
	$\mathcal{M}(W_{H}^{s,2}(\R^n))\subset \mathcal{M}(W_{H}^{t,2}(\R^n))$ if $s\geq t\geq 0$. In particular, 
	$\mathcal{M}(W_{H}^{s,2}(\R^n))\subset L^{\infty}(\RN)$ if $s\geq 0.$
\label{prop7}
\end{lem}

\begin{proof}
Let  $C_{0}^{\infty}(\R^n)$ be the space of smooth functions with compact support on $\R^n$.
Suppose $m\in \mathcal{M}(W_{H}^{s,2}(\RR))$, then for any positive integer $N$ and any $f\in C_{0}^{\infty}(\R^n)$ we have
$$
\|m^{N}f\|_{L^{2}(\RR)}^{\frac{1}{N}}\leq \|m^{N}f\|_{W_{H}^{s,2}(\RR)}^{\frac{1}{N}}\leq\|m\|_{\mathcal{M}(W_{H}^{s,2}(\RR))}\|f\|_{W_{H}^{s,2}(\RR)}^{\frac{1}{N}}.
$$
In addition, for $\Omega:=\{x\in \RR:\ |m(x)|>\|m\|_{\mathcal{M}(W_{H}^{s,2}(\RR))}+1\}$ we have
$$
	\|m^{N}f\|_{L^{2}(\RR)}^{\frac{1}{N}} \geq (\|m\|_{\mathcal{M}(W_{H}^{s,2}(\RR))}+1) \Big( \int_\Omega |f(x)|^2dx \Big)^{1\over 2N}.
$$

We first claim $m\in L^\infty(\R^n)$. If it's not true, then $\Omega$ 
is of positive measure and hence one can find a proper function $f\in C_{0}^{\infty}(\R^n)$ such that $\int_\Omega |f(x)|^2dx> 0$. Let $N\to\infty$, 
then there is a obvious contradiction.

The left of the proof is to use interpolation to the operator of multiplication by $m$ between $W_{H}^{s,2}(\RR)$ and $L^2(\RR)$.
\end{proof}

For $s\ge0$,  let $W^{s,2}(\RR)$ be the classical Sobolev space, i.e., 
$$
	W^{s,2}({\RR})=\big\{f:  \|f\|_{W^{s,2}(\RR)}= \|(1+|\cdot|^2)^{s\over 2}\mathcal{F}f\|_{L^2(\RR)}<\infty\big\}.
$$

\begin{lem}
For $s\ge0$, we have that 
$W_{H}^{s,2}(\RR)\subset W^{s,2}(\RR)$. In particular, if $s>n/2$, then
$W_{H}^{s,2}(\RR)\subset L^{\infty }(\mathbb{R}^{n})$.
\label{prop9}
\end{lem}

\begin{proof}
Let $f\in W_{H}^{s,2}(\R^n)$. It suffices 
to show that $(1+|\xi|^2)^{s\over 2}\mathcal{F}f(\xi)$ is an $L^2(\R^n)$ function. 
Actually we know $f$ belongs to $W_{H}^{s^\prime,2}(\R^n)$ for all $0\leq s^\prime\leq s$, then so does 
$\mathcal{F}f$ since the Fourier transform is an isometry on $W_H^{s^\prime,2}(\R^n)$. Hence it follows from 
Lemma \ref{lem6} that
$$
	|\xi|^{s^\prime}\mathcal{F}f(\xi)\in L^2(\R^n) \quad \text{for all}\quad 0\leq s^\prime\leq s,
$$
which implies $(1+|\xi|^2)^{s\over 2}\mathcal{F}f(\xi)\in L^2(\R^n)$. This shows $f\in W^{s,2}(\RR)$.
Thus $W_{H}^{s,2}(\RR)\subset W^{s,2}(\RR)$.

As for the $s>n/2$ case, by Sobolev embedding theorem we know $W^{s,2}(\R^n)\subset L^{\infty }(\mathbb{R}^{n})$.
\end{proof}

\begin{lem}[Leibniz's Rule]
Let $s\geq 0$. Then for any $f,g\in \widetilde{W}_H^{s,2}(\RR)$ and any integer $k\leq s$, we have the generalized Leibniz's rule
\begin{align*}
&H_{j_{1}}\cdots H_{j_{k}}(fg)\\
&=\sum_{h+l+m= k}\sum_{1\leq |j'_{1}|\leq n,\cdots ,1\leq |j'_{l}|\leq n}
\sum_{1\leq |j''_{1}|\leq n,\cdots ,1\leq |j''_{m}|\leq n}\\
&\qquad\qquad\qquad\qquad p_{h}(x) 
 (H_{j'_{1}}\cdots H_{j'_{l}}f)(H_{j''_{1}}\cdots H_{j''_{m}}g)\\
&=\sum_{|\alpha |+m= k}
\sum_{1\leq|j'_{1}|\leq n,\cdots ,1\leq |j'_{m}|\leq n}(\frac{\partial^{\alpha } }{\partial x^{\alpha }}f)(H_{j'_{1}}\cdots H_{j'_{m}}g)\\
&=\sum_{h+|\alpha |+|\beta |=k}p_{h}(\frac{\partial^{\alpha } }{\partial x^{\alpha }}f)(\frac{\partial^{\beta } }{\partial x^{\beta }}g),
\end{align*}
where $p_{h}$ is a polynomial in $x$ of order $h$.
\label{prop8}
\end{lem}

\begin{proof}
We only give the proof  for the first equality since   the other two equalities  can be
obtained by a minor modifications with it.

 For any $1\leq |j|\leq n,$
\begin{align*}
H_{j}(fg)&=\pm\frac{\partial }{\partial x_{|j|}}(fg)+x_{|j|}fg\\
&=\pm[(\frac{\partial }{\partial x_{|j|}}f)g+f(\frac{\partial }{\partial x_{|j|}}g)]+x_{|j|}fg\\
&=[\pm(\frac{\partial }{\partial x_{|j|}}f)+x_{|j|}f]g+f[\pm(\frac{\partial }{\partial x_{|j|}}g)+x_{|j|}g]-x_{|j|}fg\\
&=(H_{j}f)g+f(H_{j}g)-x_{|j|}fg.
\end{align*}
Assume the first equality holds for any $k'<k$, and for convenience, we omit the subscripts of the sum.
Then for any $1\leq |j|\leq n$,
\begin{align*}
 &H_{j}H_{j_{1}}\cdots H_{j_{k'}}(fg) \\
&=H_{j}\sum 
p_{h}(x)\cdot (H_{j'_{1}}\cdots H_{j'_{l}}f)(H_{j''_{1}}\cdots H_{j''_{m}}g)]\\
&=
\sum H_{j}[p_{h}(x)\cdot (H_{j'_{1}}\cdots H_{j'_{l}}f)(H_{j''_{1}}\cdots H_{j''_{m}}g)]\\
&=
\sum \pm\frac{\partial }{\partial x_{|j|}}[p_{h}(x)\cdot (H_{j'_{1}}\cdots H_{j'_{l}}f)(H_{j''_{1}}\cdots H_{j''_{m}}g)]\\
&\quad+ x_{|j|}p_{h}(x)\cdot (H_{j'_{1}}\cdots H_{j'_{l}}f)(H_{j''_{1}}\cdots H_{j''_{m}}g)\\
&=
\sum [\pm\frac{\partial }{\partial x_{|j|}}p_{h}(x)]\cdot (H_{j'_{1}}\cdots H_{j'_{l}}f)(H_{j''_{1}}\cdots H_{j''_{m}}g)\\
&\quad+p_{h}(x)\cdot [\pm\frac{\partial }{\partial x_{|j|}}(H_{j'_{1}}\cdots H_{j'_{l}}f)](H_{j''_{1}}\cdots H_{j''_{m}}g)\\
&\quad+ p_{h}(x)\cdot (H_{j'_{1}}\cdots H_{j'_{l}}f)[\pm\frac{\partial }{\partial x_{|j|}}(H_{j''_{1}}\cdots H_{j''_{m}}g)]\\
&\quad+ x_{|j|}p_{h}(x)\cdot (H_{j'_{1}}\cdots H_{j'_{l}}f)(H_{j''_{1}}\cdots H_{j''_{m}}g)\\
&=
\sum [p_{h-1}(x)+ x_{|j|}p_{h}(x)] \cdot (H_{j'_{1}}\cdots H_{j'_{l}}f)(H_{j''_{1}}\cdots H_{j''_{m}}g)\\
&\quad+p_{h}(x)\cdot [(H_{j}-x_{|j|})(H_{j'_{1}}\cdots H_{j'_{l}}f)](H_{j''_{1}}\cdots H_{j''_{m}}g)\\
&\quad+ p_{h}(x)\cdot (H_{j'_{1}}\cdots H_{j'_{l}}f)[(H_{j}-x_{|j|})(H_{j''_{1}}\cdots H_{j''_{m}}g)]\\
&=
\sum [p_{h-1}(x)-p_{h+1}(x)]\cdot (H_{j'_{1}}\cdots H_{j'_{l}}f)(H_{j''_{1}}\cdots H_{j''_{m}}g)\\
&\quad+ p_{h}(x)\cdot (H_{j}H_{j'_{1}}\cdots H_{j'_{l}}f)(H_{j''_{1}}\cdots H_{j''_{m}}g)\\
&\quad+ p_{h}(x)\cdot (H_{j'_{1}}\cdots H_{j'_{l}}f)(H_{j}H_{j''_{1}}\cdots H_{j''_{m}}g),
\end{align*}
where  $h+l+m+1=k^\prime+1$. By mathematical induction, the proof of the first equality is complete. 
The proof of Lemma~\ref{prop8} is complete. 
\end{proof}

\begin{prop}
Let $k>n/2$ be an integer. Then we have

 \begin{itemize}
\item[(i)] 
$W_{H}^{k,2}(\RR)$ is an algebra and
$W_{H}^{k,2}(\RR)\subset \mathcal{M}(W_{H}^{k,2}(\RR))$.

\item[(ii)]   $W^{k,2}(\RR)\subset \mathcal{M}(W_{H}^{k,2}(\RR))$.

\item[(iii)]  $\mathcal{M}(W^{k,2}(\RR))\subset \mathcal{M}(W_{H}^{k,2}(\RR))$.
\end{itemize}
\label{thm10}
\end{prop}

\begin{proof}
Since $W_{H}^{k,2}(\RR)=\widetilde{W}_H^{k,2}(\RR)$ by Lemma \ref{lem1}, we work with functions $f, g\in \widetilde{W}_H^{k,2}(\RR)$.
By the Leibniz's rule, for all $0\leq k^\prime\leq k$, we have
\begin{align*}
&H_{j_{1}}\cdots H_{j_{k^\prime}}(fg)\\
&=\sum_{h+l+m= k^\prime}\sum_{1\leq |j'_{1}|\leq n,\cdots ,1\leq |j'_{l}|\leq n}
\sum_{1\leq |j''_{1}|\leq n,\cdots ,1\leq |j''_{m}|\leq n}p_{h}(x) \\
&\qquad\qquad \cdot(H_{j'_{1}}\cdots H_{j'_{l}}f)(H_{j''_{1}}\cdots H_{j''_{m}}g).
\end{align*}
It follows from Lemma \ref{lem6} and Lemma \ref{prop9} that
$$p_{h}H_{j'_{1}}\cdots H_{j'_{l}}f\in \widetilde{W}_H^{k-(h+l),2}(\RR)\subset W^{k-(h+l),2}(\RR)$$
and
$$H_{j''_{1}}\cdots H_{j''_{m}}g\in \widetilde{W}_H^{k-m,2}(\RR)\subset W^{k-m,2}(\RR).$$

Let $f_0= p_{h}H_{j'_{1}}\cdots H_{j'_{l}}f$ and $g_0= H_{j''_{1}}\cdots H_{j''_{m}}g$. 
If $k-(h+l)>\frac{n}{2}$ or $k-m>\frac{n}{2}$, then $f_0g_0\in L^2$ since one of $f_0$ and $g_0$ 
is bounded. If $k-(h+l),k-m\leq \frac{n}{2}$, then by Sobolev 
embedding theorem we know $f_0\in L^q$ for ${1\over 2}-{k-(h+l)\over n}\leq {1\over q}< {1\over 2}$ and $g_0\in L^r$ for 
${1\over 2}-{k-m\over n}\leq {1\over r}< {1\over 2}$. Since $h+l+m=k^\prime\leq k$, we can choose $q$ and $r$ such that 
${1\over 2}= {1\over q}+ {1\over r}$, hence by  H\"older's inequality, 
\begin{align*}
	\|f_0g_0\|_{L^2(\RR)}\leq \|f_0\|_{L^q} \|g_0\|_{L^r(\RR)}
	&\leq C\|f_0\|_{W^{k-(h+l),2}(\RR)} \|g_0\|_{W^{k-m,2}(\RR)}\\
	&\leq C\|f\|_{\widetilde{W}_H^{k,2}(\RR)}\|g\|_{\widetilde{W}_H^{k,2}(\RR)}.
\end{align*}
This shows $fg\in \widetilde{W}_H^{k,2}(\RR)$ and $f,g\in \mathcal{M}(\widetilde{W}_H^{k,2}(\RR))$, which implies that $W_{H}^{k,2}(\RR)= \widetilde{W}_H^{k,2}(\RR)$ is an algebra.
This proves (i).

\medskip
We now prove (ii). Let $f\in W^{k,2}(\RR)$. For any $g\in \widetilde{W}_H^{k,2}(\RR)$ we have
$$
H_{j_{1}}\cdots H_{j_{k}}(fg)
=\sum_{|\alpha |+m= k}
\sum_{1\leq|j'_{1}|\leq n,\cdots ,1\leq |j'_{m}|\leq n}(\frac{\partial^{\alpha } }{\partial x^{\alpha }}f)(H_{j'_{1}}\cdots H_{j'_{m}}g)
$$
by Lemma~\ref{prop8}. Similar to the proof above, we see that
$$
(H_{j'_{1}}\cdots H_{j'_{m}}g)\in \widetilde{W}_H^{k-m,2}(\RR)\subset W^{k-m,2}(\RR),
$$
and
$$\frac{\partial^{\alpha } }{\partial x^{\alpha }}f\in W^{k-|\alpha|,2}(\RR)= W^{m,2}(\RR).
$$
Thus $H_{j_{1}}\cdots H_{j_{k}}(fg)\in L^{2}(\mathbb{R}^{n})$.
Similarly, for any $k'\leq k$, we have $H_{j_{1}}\cdots H_{j_{k'}}(fg)\in L^{2}(\mathbb{R}^{n})$.
Hence $fg \in \widetilde{W}_H^{k,2}(\RR)=W_{H}^{k,2}(\RR)$.

Now we turn to prove (iii). Let ${\widetilde \eta}\in C_0^{\infty}$ such that ${\widetilde \eta}(x)=1 $ if $|x|\leq 2$; $0$ if $|x|\geq 3$. 
 Denote by ${\widetilde \eta}_m(x)= {\widetilde \eta}(x+m)$. 
 Let $f\in \mathcal{M}(W^{k,2}(\RR))$ and $g\in W_{H}^{k,2}(\RR)$. Since ${\widetilde \eta}_m\in W^{k,2}(\RR)$, we know from (ii) that
$$
	f{\widetilde \eta}_m\in W^{k,2}(\RR)\subset \mathcal{M}(W_{H}^{k,2}(\RR)).
$$
Then it follows from Proposition~\ref{prop2.9} that
\begin{align*}
	\|fg\|^2_{W_{H}^{k,2}(\RR)}
&\leq C\sum_{m\in \mathbb{Z}_n} \|fg\eta_m{\widetilde \eta}_m\|_{W_{H}^{k,2}(\RR)}^2  \\
&\leq C\sum_{m\in \mathbb{Z}_n} \|f{\widetilde \eta}_m\|^2_{W^{k,2}(\RR)} \|g \eta_m\|^2_{W_{H}^{k,2}(\RR)}\\
&\leq C\|f\|_{\mathcal{M}(W^{k,2}(\RR))}^2 \sum_{m\in \mathbb{Z}_n} \|g\eta_m\|^2_{W_{H}^{k,2}(\RR)}\\
&\leq C\|f\|_{\mathcal{M}(W^{k,2}(\RR))}^2 \|g\|^2_{W_{H}^{k,2}(\RR)}
\end{align*}
This shows $f\in \mathcal{M}(W_{H}^{k,2}(\RR))$.

The proof of Proposition \ref{thm10} is complete.
\end{proof}


In the end of this section we 
study the multipliers on the Hermite-Sobolev  spaces by localizations. 
Let   $\eta$ be a function as 
in Proposition~\ref{prop2.9}.  Recall  that   for the classical  Sobolev spaces $W^{s,2}(\RN)$,
$$
\|\eta(x+m)\|_{W^{s,2}(\RR)}=\|\eta(x)\|_{W^{s,2}(\RR)} 
$$
for every $m\in {\mathbb Z}_n$.
In \cite[Corollary 3.3]{S1},    Strichartz proved the following well-known result that   
  $f\in \mathcal{M}(W^{s,2}(\RR))$ if and only if 
$ f(x)\eta(x+m)\in \mathcal{M}(W^{s,2}(\RR))$ for all $m\in {\mathbb Z}_n$ and
$$
\sup_{m\in {\mathbb Z}_n}\|f\eta_m\|_{\mathcal{M}(W^{s,2}(\RR))}<\infty.
$$
The supremum is equivalent to $\|f\|_{\mathcal{M}(W^{s,2}(\RR))}.$
From  Theorem 2.1 and  Corollary 2.2 in \cite{S1}, we know that for $s> n/2$,  $W^{s,2}(\RR)\subseteq \mathcal{M}(W^{s,2}(\RR))$,
and so    $f\in \mathcal{M}(W^{s,2}(\RR))$ if and only if $f$ is a 
  uniformly local function in the sense of norms in $W^{s,2} $, i.e.,
   $\|f\eta_m\|_{W^{s,2}(\RR)}\leq C$ for all  $m\in {\mathbb Z}_n$.

Turning to the Hermite-Sobolev spaces, we have 
\begin{align*}
&\|\eta (x+m)\|_{\widetilde{W}_H^{k,2}(\RR)}\\
&=\sum_{1\leq |j|\leq n}\|H_{j}\eta (x+m)\|_{L^{2}(\RR)}
+\|\eta (x+m)\|_{L^{2}(\RR)}\nonumber\\
&\geq \sum_{1\leq |j|\leq n}\bigg(\|x_{|j|}\eta (x+m)\|_{L^{2}(\RR)}-\| \frac{\partial }{\partial x_{|j|}}\eta \|_{L^{2}(\RR)}\bigg)-\|\eta \|_{L^{2}(\RR)}\nonumber\\
&=  \sum_{1\leq |j|\leq n}\bigg(\|(x_{|j|}-m_{|j|})\eta (x)\|_{L^{2}(\RR)}-\| \frac{\partial }{\partial x_{|j|}}\eta\|_{L^{2}(\RR)}\bigg)\nonumber\\
&\qquad-\|\eta \|_{L^{2}(\RR)}\nonumber\\
&\geq \sum_{1\leq |j|\leq n}|m_{|j|}|\|\eta (x)\|_{L^{2}(\RR)}-\sum_{1\leq |j|\leq n}
\|x_{|j|}\eta (x)\|_{L^{2}(\RR)}\nonumber\\
&\qquad-\sum_{1\leq |j|\leq n}\| \frac{\partial }{\partial x_{|j|}}\eta \|_{L^{2}(\RR)}-\|\eta \|_{L^{2}(\RR)}\nonumber\\
&\rightarrow  \infty \hskip10mm \mbox{as}\hskip 5mm |m|\rightarrow \infty .
\end{align*}
This shows although $1\in \mathcal{M}(W_{H}^{s,2}(\RR))$, $\|1\cdot\eta_m\|_{W_{H}^{s,2}(\RR)}$ cannot be controlled by a constant, which is different from the case of Sobolev multipliers. 
We can also find a function which is not a multiplier of  $W^{s,2}(\RR)$ and not uniformly local in the sense of norms in $W_{H}^{s,2}(\RR)$, but it is a multiplier of $W_{H}^{s,2}(\RR)$. In fact, if $h(x)=e^{ix^{\frac43}}$, then $h$ is not a uniformly
local function  but $h$ is a multiplier of $W_{H}^{1,2}(\mathbb R)$. To see this, note that
for any $f\in W_{H}^{1,2}(\mathbb R)$ we have
\begin{eqnarray*}
H(hf)&=&(hf)'+x(hf)\\
&=&\frac{4}{3}x^{\frac{1}{3}}e^{ix^{\frac43}}f+e^{ix^{\frac{4}{3}}}f'+xe^{ix^{\frac{4}{3}}}f\\
&=&\left(\frac{4}{3}x^{\frac{1}{3}}+x\right)e^{ix^{\frac{4}{3}}}f+e^{ix^{\frac{4}{3}}}f'.
\end{eqnarray*}
Since $f\in W_{H}^{1,2}(\mathbb R)$, we see that $xf\in L^{2}(\mathbb{R})$. Furthermore,
$$\left(\frac{4}{3}x^{\frac{1}{3}}+x\right)e^{ix^{\frac{4}{3}}}f\in L^{2}(\mathbb{R}).$$
Thus $H(hf)\in L^{2}(\mathbb{R}).$ This means that $h\in  \mathcal{M}(W_{H}^{1,2}(\mathbb R))$. However,
it is not difficult to see that $\|h\eta_{m}\|_{W_{H}^{1,2}(\mathbb R)}\rightarrow\infty$ as $m\to\infty$. We see easily that $h$ is not a multiplier of  $W^{1,2}$  since $\|h\eta_{m}\|_{W^{1,2}(\mathbb R)}\rightarrow \infty $ as $m\rightarrow \infty .$

To obtain multipliers on the  Hermite-Sobolev spaces, we have the following proposition.

 \begin{prop}\label{lem 2.14}
  Let $s\geq 0$. Then 
  $f\in {\mathcal M}\left(W^{s, 2}_H(\RN)\right)$ if and only if $f\eta_m \in {\mathcal M}\left(W^{s, 2}_H(\RN)\right)$
 uniformly in $m\in{\mathbb Z_n}$, i.e., 
 $$
 \sup_{m\in{\mathbb Z_n}}   \|f\eta_m\|_{{\mathcal M}\left(W^{s, 2}_H(\RN)\right)}   <\infty.
 $$
 The supremum is equipped to $\|f\|_{{\mathcal M}\left(W^{s, 2}_H(\RN)\right)}.$ 
 \end{prop}
 
 \begin{proof}
 Let $f\in {\mathcal M}\left(W^{s, 2}_H(\RN)\right)$. Then for every $g\in W^{s,2}_H(\RN),$
 \begin{eqnarray*}  
 \|f\eta_m  g\|_{W^{s, 2}_H(\RN)}&\leq &C  \| \eta_m  g\|_{W^{s, 2}_H(\RN)}
 \leq   C 
 \|g\|_{W^{s, 2}_H(\RN)} 
\end{eqnarray*}
 with a constant $C$ independent of $m$, where  in the second inequality 
we used    Proposition~\ref{prop2.9}.
 
 \medskip
 
 Conversely, assume that  
 $$
  \sup_{m\in{\mathbb Z_n}}  \|f \eta_m\|_{{\mathcal M}\left(W^{s, 2}_H(\RN)\right)}   =: M<\infty.
 $$
Let $\eta\in C_0^{\infty}(\RR)$ such that ${\widetilde \eta}(x)=1$ on the cube $\{ |x|\leq 2\}$ 
and   ${\widetilde \eta}_m(x)= {\widetilde \eta}(x+m)$. 
We follow an argument as in Proposition~\ref{prop2.9} to show that  for every $g\in W^{s, 2}_H(\RN),$
   \begin{eqnarray*}  
 \|fg\|_{W^{s, 2}_H(\RN)}^2
 &\leq& C \sum_{m\in {\mathbb Z}_n} \| fg\eta_m {\widetilde \eta}_m\|_{W^{s, 2}_H(\RN)}^2 \\
 &\leq& CM^2 \sum_{m\in {\mathbb Z}_n} \|  g{\widetilde \eta}_m\|_{W^{s, 2}_H(\RN)}^2   \\
  &\leq&  CM^2 \|  g\|_{W^{s, 2}_H(\RN)}^2.
\end{eqnarray*} 
The proof of Proposition \ref{lem 2.14} is complete.
 \end{proof}

 As a consequence of Proposition~\ref{lem 2.14}, we have the following result.

 \begin{cor}\label{prop 2.15} Let $k>n/2$ be an integer. If 
  $$
 \sup_{m\in\mathbb{Z}_n} \left\|f\eta_m\right\|_{W^{k, 2}_H(\RN)}<\infty,  
  $$
then  $f\in {\mathcal M}(W^{k, 2}_H(\RN))$. 

 \end{cor}
 
 \begin{proof}
 It follows by Proposition~\ref{thm10} that for an integer $k>n/2$,
  we have that $W^{k, 2}_H(\RN)\subset {\mathcal M}(W^{k, 2}_H(\RN)),
 $ 
 and thus there exists some $M>0$ such that for every $m\in {\mathbb Z_n},$
 \begin{eqnarray*}
 \left\|f \eta_m\right\|_{{\mathcal M}(W^{k, 2}_H(\RN))} 
  &\leq& 
  \left\|f \eta_m\right\|_{W^{k, 2}_H(\RN)} 
   \leq 
  M.
\end{eqnarray*}
 Then we apply Proposition~\ref{lem 2.14} to obtain that $f\in {\mathcal M}(W^{k, 2}_H(\RN)).$
 \end{proof}

\begin{rem}
It would be interesting to establish  a necessary and sufficient condition for  $f$ to be in $\mathcal{M}(W^{s,2}_H(\RR))$
 $s>0$. To the best of our knowledge, it is not clear for us yet.
 \end{rem}

  \medskip

\section{Proof of Theorem~\ref{th1.1}}

In the following, for any $N>0$, we use $Q_{N}$ to denote the
cube centered at $0\in\mathbb{R}^{n}$ with side length $2N$. Let $\Delta =\cup \Delta_{l}$ be an
arbitrary partition of $Q_{N}$ and choose $x_{l}\in \Delta_{l}$ for each $l$. 
Suppose that  $f$ is a measurable function on $\mathbb{R}^{n}$.  We define the Riemann sum of $f$ as
$$
S_{\Delta }^{N}(f)=\sum_{l}f(x_{l})|\Delta_{l}|,
$$
where $|\Delta_{l}|$ denotes the volume of $\Delta_{l}$. Let $\text{diam}\,(\Delta_l)$
 denote the diameter of $\Delta_l$ and $\lambda:= \max\limits_l \text{diam}\,(\Delta_l)$. If
 $\lim\limits_{N\rightarrow\infty }\lim\limits_{\lambda \rightarrow 0}S_{\Delta }^{N}(f)$ 
exists, we say that $f$ is integrable on $\mathbb{R}^{n}$ and we write
\begin{eqnarray*}
\int_{\mathbb{R}^{n}}f\,dx&=&\lim_{N\rightarrow\infty}\int_{Q_N}f\,dx\\
&=&\lim\limits_{N\rightarrow\infty }\lim\limits_{\lambda \rightarrow 0}S_{\Delta }^{N}(f).
\end{eqnarray*}

\begin{lem}\label{prop15}
	Let $s\geq 0$ and $T$ be a bounded operator on $W_H^{s,2}(\R^n)$, which commutes with translations $\tau_a$ 
	for all $a\in\RR$. Then for $f,g\in C_0^\infty(\R^n)$, we have 
$$
	T(f\ast g)= Tf\ast g= f\ast Tg.
$$
\end{lem}

\begin{proof}
Let $f,g\in C_0^\infty(\R^{n})$. 
It follows from the proof of  Theorem 2.3.20 in \cite{G1} 
that  $S_{\Delta }^N (f,g) \rightarrow f\ast g$ in the Schwarz space $\mathscr{S}$. This implies 
$$
	H_{j_{1}}\cdots H_{j_{k^\prime}}S_{\Delta }^N (f,g) \rightarrow H_{j_{1}}\cdots H_{j_{k^\prime}}
	 (f\ast g) \hskip 2mm \text{in}\hskip 2mm\ L^\infty \hskip 2mm \text{for}\hskip 2mm0\leq k^\prime\leq k\in \mathbb{N}.
$$
Since $f$ and $g$ have compact supports, we know
$$
	H_{j_{1}}\cdots H_{j_{k^\prime}}S_{\Delta }^N (f,g) \rightarrow H_{j_{1}}\cdots H_{j_{k^\prime}} (f\ast g) \qquad \text{in}\ L^2(\mathbb R^n),
$$
which means that $S_{\Delta }^N(f,g)\rightarrow f\ast g$ in $W_{H}^{k,2}(\mathbb R^n)$.
 Thus $S_{\Delta }^N(f,g)\rightarrow f\ast g$ in $W_{H}^{k,2}(\mathbb R^n)\subset W_{H}^{s,2}(\mathbb R^n)$ if one let $k= [s]+1$.

Since $T$ is bounded on $W_H^{s,2}(\R^n)$ and commutes with translations, we have
\begin{eqnarray*}
	T(f\ast g)(x)&=&T(\lim\limits_{N\rightarrow\infty }\lim\limits_{\lambda \rightarrow 0} S_{\Delta }^N(f,g))(x)\\
&=& \lim\limits_{N\rightarrow\infty }\lim\limits_{\lambda \rightarrow 0} T(S_{\Delta }^N(f,g))(x) \\
&=& \lim\limits_{N\rightarrow\infty }\lim\limits_{\lambda \rightarrow 0} \sum_l f(y_l)Tg(x-y_l)|\Delta_l|.
\end{eqnarray*}
Note $f\in C_0^\infty(\R^{n})$ and $Tg\in L^2(\mathbb R^n)$, which shows $f\ast g\in L^2(\mathbb R^n)$, i.e., the integral defining the convolution of $f$ and $g$ converges. So
$$
	\lim\limits_{N\rightarrow\infty }\lim\limits_{\lambda \rightarrow 0} \sum_l f(y_l)Tg(x-y_l)|\Delta_l|= f\ast Tg(x)\quad
$$
pointwisely in $x$. This shows $T(f\ast g)= f\ast Tg$.
\end{proof}

If $\mathcal{F}(f)$ denotes the Fourier transformation of $f$, then for $f,g\in C_c^\infty (\R^{n})$,
\begin{equation}\label{eq3}
\mathcal{F}(f)\mathcal{F}(Tg)=\mathcal{F}(Tf)\mathcal{F}(g).
\end{equation}

\begin{prop}
Let $s\geq 0$. Suppose $T$ is a bounded operator on $W_{H}^{s,2}(\RR)$. If $T$ commutes with all translations $\tau_{a}$,
$a\in \mathbb{R}^{n}$, on $W_{H}^{s,2}(\RR)$, then there is an $m\in \mathcal{M}(W_{H}^{s,2}(\RR))$ such that
\begin{equation}\label{eq4}
\mathcal{F}(Tf)=m\mathcal{F}(f),\qquad f\in W_{H}^{s,2}(\RR).
\end{equation}
Conversely, for any $m\in \mathcal{M}(W_{H}^{s,2}(\RR))$, $T=\mathcal{F}^{-1}M_{m}\mathcal{F}$ is bounded on $W_{H}^{s,2}(\RR)$
and commutes with translations on $\mathscr{S}(\RR)$, where $M_{m}f=mf$ for any $f\in W_{H}^{s,2}(\RR)$.
\label{thm16}
\end{prop}

\begin{proof}
For any $f\in W_{H}^{s,2}(\R^n)$ there is a sequence $\{f_{j}\}\subset C_0^\infty(\R^{n})$, such that
$$
	\|f_{j}-f\|_{W_{H}^{s,2}(\RR)}\rightarrow 0, \quad j\to\infty.
$$
Since $T$ is bounded on $W_{H}^{s,2}(\R^n)$, we see that
$$
	\|Tf_{j}-Tf\|_{W_{H}^{s,2}(\RR)}\rightarrow 0, \quad j\to\infty.
$$
Consequently,
$$
	\|\mathcal{F}(f_{j})-\mathcal{F}(f)\|_{L^2(\RR)},\ \|\mathcal{F}(Tf_{j})-\mathcal{F}(Tf)\|_{L^2(\RR)}\rightarrow 0, \quad j\to\infty.
$$

Then we can find subsequences of $\{\mathcal{F}(f_{j})\}$ and $\{\mathcal{F}(Tf_{j})\}$, 
which are still denoted by $\{\mathcal{F}(f_{j})\}$ and $\{\mathcal{F}(Tf_{j})\}$ respectively,
 such that $\mathcal{F} (f_{j})\rightarrow \mathcal{F}(f)$ a.e. and $ \mathcal{F}(Tf_{j})\rightarrow \mathcal{F}(Tf)$ 
 a.e.. By (\ref{eq3}), we see that for any $g\in C_0^\infty(\R^{n})$,
$$
	\mathcal{F}(f_j)\mathcal{F}(Tg)=\mathcal{F}(Tf_j)\mathcal{F}(g).
$$
Let $j\rightarrow\infty$, we have
$$
	\mathcal{F}(f)\mathcal{F}(Tg)=\mathcal{F}(Tf)\mathcal{F}(g)\hskip 5mm a.e.
$$
By the same token, there still holds
$$
	\mathcal{F}(f)\mathcal{F}(Tg)=\mathcal{F}(Tf)\mathcal{F}(g)\hskip 5mm a.e.
$$
for all $f,g\in W_H^{s,2}(\R^n)$. We may choose some $g\in W_{H}^{s,2}(\R^n)$ such that $\mathcal{F}(g)$ has no zeros
on $\mathbb{R}^{n}$ and let $m=\mathcal{F}(Tg)/\mathcal{F}(g)$. Then
$$
\mathcal{F}(Tf)=m\mathcal{F}(f),\qquad f\in W_{H}^{s,2}(\R^n).
$$
Note that
$$
\|\mathcal{F}(Tf)\|_{ W_{H}^{s,2}(\RR)}=\|Tf\|_{ W_{H}^{s,2}(\RR)}\leq \|T\|\|f\|_{ W_{H}^{s,2}(\RR)}=
\|T\|\|\mathcal{F}(f)\|_{ W_{H}^{s,2}(\RR)}.
$$
Thus
$$
\|m\mathcal{F}(f)\|_{W_{H}^{s,2}(\RR)}\leq \|T\|\|\mathcal{F}(f)\|_{ W_{H}^{s,2}(\RR)}.
$$
This shows that $m\in \mathcal{M}( W_{H}^{s,2}(\R^n))$.

Conversely, if $m\in \mathcal{M}( W_{H}^{s,2}(\R^n))$, then $T$ defined by (\ref{eq4}) is bounded.

The proof of Proposition \ref{thm16} is complete.
\end{proof}

\begin{prop}
Let $s\geq 0.$  Suppose that $T$ is a bounded operator on $W_{H}^{s,2}(\RR)$ and it
commutes with all translations $\tau_{a}$,
$a\in \mathbb{R}^{n}$, on $W_{H}^{s,2}(\RR)$. 
  If
$$\mathcal{F}(Tf)(x)= m(x)\mathcal{F}(f)(x), \qquad f\in W_{H}^{s,2}(\RR),$$
with an $m\in \mathcal{M}(W_{H}^{s,2}(\RR))$, then for every $g\in F^{s,2}(\mathbb{C}^{n})$,
$$
\mathcal{B}T\mathcal{B}^{-1}g(z)= \int_{\mathbb{C}^{n}}
g(w)e^{z\cdot \bar{w}}\varphi(z-\bar{w})\,d\lambda(w),\qquad z\in\mathbb{C}^{n},
$$
where
$$
\varphi(z)= \left({2\over \pi}\right)^{n\over2} \int_{\mathbb{R}} m(x)e^{-2(x-\frac{i}{2}z)^2} dx\in F^{s,2}(\CC).
$$
\label{thm18}
\end{prop}

\begin{proof}
Following an argument of Lemma 3.4 in \cite{CJSWY}, we  obtain 
$$\varphi(z)= \left({2\over \pi}\right)^{n\over2} \int_{\mathbb{R}} m(x)e^{-2(x-\frac{i}{2}z)^2} dx 
$$ 
in terms of  $m\in  \mathcal{M}(W_{H}^{s,2}(\RR))$.
 By Proposition~\ref{thm16}, it suffices to show that 
 $\varphi(z) \in F^{s,2}(\CC)$ 
for $m\in\mathcal{M}(W_{H}^{s,2}(\RR))$. To show it,   for $z\in\mathbb{C}^{n}$ we  write $z=u+iv$, and the key observation is 
the following:
$$
	\varphi(z)= C\mathcal{F}^{-1}[m(x-\frac{1}{2}v)e^{-2x^{2}}](u)e^{|u|^2\over 2}.
$$
Notice that 
$$
	\int_{\mathbb{R}^{n}} (1+|\xi|)^{2s}|\mathcal{F}^{-1} f(\xi)|^2 d\xi\leq C\|f\|_{\mathcal{W}^{s,2}}^2\leq C\|f\|_{W_H^{s,2}(\RR)}^2,
$$
and from  \eqref{ppp},
\begin{align*}
	&\|m(x-\frac{1}{2}v)e^{-2x^{2}}\|_{W_H^{s,2}(\RR)}^2 \\
	&\leq \big(1+{|v|\over2}\big)^{2s}\|m(x)e^{-2(x+\frac12 v)^{2}}\|_{W_H^{s,2}(\RR)}^2\\
&\leq \|m\|_{\mathcal{M}(W_H^{s,2}(\RR))}^2 \big(1+{|v|\over2}\big)^{2s} \|e^{-2(x+\frac12 v)^{2}}\|_{W_H^{s,2}(\RR)}^2 \\
&\leq \|m\|_{\mathcal{M}(W_H^{s,2}(\RR))}^2 \big(1+{|v|\over2}\big)^{4s} \|e^{-2x^{2}}\|_{W_H^{s,2}(\RR)}^2. 
\end{align*}
By Lemma~\ref{lem3}, we have
\begin{align*}
	\|\varphi \|_{F^{s,2}(\CC)}^{2}
&\leq C \int_{\mathbb{C}^{n}}(1+|z|)^{2s}|\varphi (z)|^{2}e^{-|z|^{2}}dz\\
&\leq C \int_{\mathbb{R}^{n}}\int_{\mathbb{R}^{n}}(1+|u|)^{2s}(1+|v|)^{2s}|\varphi (u+iv)|^2
e^{-(|u|^{2}+|v|^{2})}dudv\\
&\leq C \int_{\mathbb{R}^{n}} (1+|v|)^{2s}e^{-|v|^2}\\
&\qquad \int_{\mathbb{R}^{n}} (1+|u|)^{2s}|\mathcal{F}^{-1}[m(x-\frac{1}{2}v)e^{-2x^{2}}](u)|^2 dudv\\
&\leq C \int_{\mathbb{R}^{n}} (1+|v|)^{2s}e^{-|v|^2} \|m(x-\frac{1}{2}v)e^{-2x^{2}}\|_{W_H^{s,2}(\RR)}^2 dv\\
&\leq C \|m\|_{\mathcal{M}(W_H^{s,2}(\RR))}^2 \int_{\mathbb{R}^{n}} (1+|v|)^{2s}\big(1+{|v|\over2}\big)^{4s} e^{-|v|^2} dv \\
&\leq C \|m\|_{\mathcal{M}(W_H^{s,2}(\RR))}^2.
\end{align*}
This proves  $\varphi \in F^{s,2}(\CC)$.  The proof  of Proposition~\ref{thm18} is complete.
\end{proof}

Finally, we are ready to prove our Theorem~\ref{th1.1}.

\begin{proof}[Proof of Theorem~\ref{th1.1}]
First, we assume that 
$$
	\varphi (z)= \left({2\over \pi}\right)^{n\over 2} \int_{\mathbb{R}^{n}}m(x)e^{-2(x-\frac{i}{2}z)^{2}}dx, \qquad z\in \mathbb{C}^{n}
$$
where $m$ is a multiplier on the space $ W_{H}^{s,2}(\RR), s\geq 0$.
Let  $S_\varphi$ be  an integral operator as in \eqref{e1.1}. To prove  that  $S_\varphi$
 is bounded on the space $F^{s,2}(\CC),$
we notice that   from Proposition~\ref{thm18},   $\varphi\in F^{s,2}(\CC)$ and
$$S_\varphi= \mathcal{B}T\mathcal{B}^{-1},
$$
where $T$ is given by 
$$
\mathcal{F}(Tf)(x) = m(x)\mathcal{F}(f)(x)\hskip 4mm \text{for all} \hskip 4mmf\in W^{s,2}_H(\RR).
$$
By Lemma~\ref{lem5},   the operator  $T$ is bounded on the space 
$W^{s,2}_H(\RR)$.  From the properties  of the operators $\mathcal{B}$ and
$\mathcal{B}^{-1}$,  we see that $S_\varphi$ is bounded on the space $F^{s,2}(\CC).$

Conversely, let $S_\varphi$ be a bounded  operator on $F^{s,2}(\CC)$  as  in \eqref{e1.1} . Then 
from the properties  of the operators $\mathcal{B}$ and
$\mathcal{B}^{-1}$, we have that 
$$T=\mathcal{B}^{-1}S_{\varphi }\mathcal{B}
$$
is bounded on $W_{H}^{s,2}(\RR)$. Note that $S_{\varphi }W_{a}f=W_{a}S_{\varphi }f$ for any $a\in \mathbb{R}^{n}$ and $f\in F^{s,2}(\CC)$. 
It follows that $T\tau_{a}f=\tau_{a}Tf$ for any $a\in\mathbb{R}^{n}$ and $f\in W_{H}^{s,2}(\RR)$.
Thus by Proposition~\ref{thm16}, there is an $ m\in \mathcal{M}(W_{H}^{s,2}(\RR))$ such that $\mathcal{F}(Tf)=m\mathcal{F}f$. 
This implies 
$$S_\varphi=\mathcal{B} \mathcal{F}^{-1} M_{m} \mathcal{F} \mathcal{B}^{-1}.
$$ 
By Proposition~\ref{thm18}, $\mathcal{B} \mathcal{F}^{-1} M_{m} \mathcal{F} \mathcal{B}^{-1}$ 
is  an integral operator $S_{\varphi_0}$, where
$$
	\varphi_{0}(z)=\left({2\over \pi}\right)^{n\over2}\int_{\mathbb{R}} m(x)e^{-2(x-\frac{i}{2}z)^2} dx\in F^{s,2}(\CC).
$$
This implies for all $g\in F^{s,2}(\C^n)$, there holds
$$
	\int_{\mathbb{C}^{n}} g(w)e^{z\cdot \bar{w}}(\varphi (z-\bar{w}) -\varphi_{0}(z-\bar{w}))d\lambda(w)=0,\qquad z\in\mathbb{C}^{n}.
$$

To finish the proof,  it suffices  to show that $\varphi =\varphi_{0}$. Taking  $z=0$ in the above equality,
 we see that for all $g\in  F^{s,2}(\mathbb{C}^{n})$,
$$
	\int_{\mathbb{C}^n} g(w) (\varphi(-\bar{w})-\varphi_0(-\bar{w}))\, d\lambda(w) =0.
$$
Write $\psi(w)=\varphi(-w)-\varphi_0(-w)\in F^{2}(\mathbb{C}^{n})$. Then $\psi$ has the series expansion
$$
	\psi(w)=\sum\limits_{\alpha} c_{\alpha} e_{\alpha} (w)=\sum\limits_{\alpha} c_{\alpha}
\left({1\over \alpha! }\right)^{1\over2}\, w^{\alpha }
$$
with  $\sum_{\alpha} |c_{\alpha}|^2 = \|\psi\|^2_{F^2(\mathbb{C}^{n})}$. Letting $g=e_\alpha$ for all $\alpha\in\mathbb{N}_{0}^{n}$,
we obtain
\begin{eqnarray*}
	c_{\alpha }&= &\int_{\mathbb{C}^n} e_{\alpha }(w)\psi({\bar w} ) d\lambda(w)\\
&= &\int_{\mathbb{C}^n} e_{\alpha }(w) (\varphi(-\bar{w})-\varphi_0(-\bar{w})) d\lambda(w)=0.
\end{eqnarray*}
This shows $\varphi=\varphi_0$. Hence, the proof of Theorem~\ref{th1.1} is complete.
\end{proof}

 \medskip

 \noindent
{\bf Acknowledgments.}
The authors would like to thank L. Yan and K. Zhu for helpful discussions.
 G.F. Cao was supported by NNSF of China (Grant Number 12071155).
L. He  was supported by NNSF of China (Grant Number  11871170). 
J. Li is supported by the Australian Research Council (ARC) through the
research grant DP170101060.

\end{document}